\documentclass{amsart}

\usepackage[utf8]{inputenc}
\usepackage[english]{babel}

\usepackage{braket}
\usepackage{amsmath}
\usepackage{amstext}
\usepackage{amssymb}
\usepackage{amsthm}
\usepackage{mathtools}
\usepackage{stmaryrd}
\usepackage{mathrsfs}
\usepackage{extarrows}

\usepackage{microtype}

\usepackage[retainorgcmds]{IEEEtrantools}
\usepackage{graphicx}
\usepackage{amsfonts}

\usepackage{units}
\usepackage{enumerate}
\usepackage{url}

\usepackage{tikz}
\usetikzlibrary{matrix,arrows,calc}
\usepackage{tikz-cd}

\usepackage{hyperref}

\usepackage[margin=1cm]{caption}

\theoremstyle{definition}
\newtheorem{theorem}{Theorem}[section]
\newtheorem{definition}[theorem]{Definition}
\newtheorem{remark}[theorem]{Remark}
\newtheorem{lemma}[theorem]{Lemma}
\newtheorem{proposition}[theorem]{Proposition}
\newtheorem{corollary}[theorem]{Corollary}
\newtheorem{conjecture}[theorem]{Conjecture}

\newcommand{\mb}[1]{\mathbb{#1}}
\newcommand{\mc}[1]{\mathcal{#1}}

\newcommand{\N}{\mb N}
\newcommand{\Z}{\mb Z}

\newcommand{\F}{\mc F}

\newcommand{\D}{\mathcal D}
\newcommand{\M}{\mathcal M}

\newcommand{\sal}{\text{Sal}}

\newcommand{\bsal}{\overline{\sal}}

\newcommand{\h}[1]{\smash{\widehat{#1}}}

\newcommand{\leqr}{\preceq_{\,\text{R}}}
\newcommand{\leql}{\preceq_{\,\text{L}}}

\newcommand{\pr}[3]{{\Pi}(#1,#2,#3)}
\newcommand{\prp}[3]{{\Pi'}(#1,#2,#3)}
\newcommand{\product}{\xi}

\DeclareMathOperator{\id}{id}
\DeclareMathOperator{\rev}{rev}

\newcommand{\xstar}{X^{(*)}}
\newcommand{\ystar}{Y^{(*)}}
\newcommand{\zstar}{Z^{(*)}}

\renewcommand{\sf}{S^f}

\usepackage[foot]{amsaddr}

\title{On the classifying space of Artin monoids}
\author{Giovanni Paolini}
\address{Scuola Normale Superiore, Piazza dei Cavalieri 7, 56126 Pisa (Italy)}
\email{giovanni.paolini@sns.it}

\begin{document}

\begin{abstract}
  A theorem proved by Dobrinskaya in 2006 shows that there is a strong connection between the $K(\pi,1)$ conjecture for Artin groups and the classifying spaces of Artin monoids.
  More recently Ozornova obtained a different proof of Dobrinskaya's theorem based on the application of discrete Morse theory to the standard CW model of the classifying space of an Artin monoid.
  In Ozornova's work there are hints at some deeper connections between the above-mentioned CW model and the Salvetti complex, a CW complex which arises in the combinatorial study of Artin groups.
  In this work we show that such connections actually exist, and as a consequence we derive yet another proof of Dobrinskaya's theorem.
\end{abstract}

\maketitle

\noindent {\bf Keywords:} Artin groups, Artin monoids, Discrete Morse theory, Coxeter groups.

\vskip0.2cm

\noindent {\bf 2000 Mathematics Subject Classification:} 20F36, 55R35, 55U10, 52C35.

\vskip0.2cm

The Version of Record of this manuscript has been published and is available in \emph{Communications in Algebra}
(12 Apr 2017) \url{http://www.tandfonline.com/10.1080/00927872.2017.1281931}.

\section{Introduction}

The beginning of the study of Artin groups dates back to the introduction of braid groups in the 20s.
Artin groups where defined in general by Tits and Brieskorn in the 60s, in relation to the theory of Coxeter groups and singularity theory.
Since then some of their general properties were studied (cf. \cite{brieskorn1972artin}) but many questions remain open.

One of these questions is the so called $K(\pi,1)$ conjecture, which states that a certain space $N(A)$ with the homotopy type of a finite CW complex is a classifying space for the corresponding Artin group $A$.
Such conjecture was proved for some particular families of Artin groups (see \cite{charney1995k, callegaro2010k, deligne1972immeubles, hendriks1985hyperplane, okonek1979dask}) but not in the general case.

Every Artin group $A$ has a special submonoid $A^+$, called Artin monoid, whose group of fractions is $A$ itself.
In 2006 Dobrinskaya \cite{dobrinskaya2006configuration} proved that the $K(\pi,1)$ conjecture holds for an Artin group $A$ if and only if the natural map $BA^+\to BA$ between the classifying spaces of $A^+$ and $A$ is a homotopy equivalence.
This revealed an interesting connection between the $K(\pi,1)$ conjecture and the geometry beneath Artin monoids.
A more concrete proof of the same theorem was given more recently by Ozornova \cite{ozornova2017discrete}, using the technique of discrete Morse theory (cf. \cite{forman1998morse, batzies2002discrete}).
This new proof gave some hints, and left some open questions, about a further connection between the standard CW model for $BA^+$ and the Salvetti complex, a finite CW model for the space $N(A)$ introduced by Salvetti \cite{salvetti1987topology, salvetti1994homotopy}.
In this work we will explore such connection, showing that the standard CW model for $BA^+$ can be collapsed (in the sense of discrete Morse theory) to obtain a CW complex which is naturally homotopy equivalent to the Salvetti complex.
This result gives another proof of Dobrinskaya's theorem.
In addition, it implies that the algebraic Morse complex constructed in \cite{ozornova2017discrete} coincides with the algebraic complex which computes the cellular homology of the Salvetti complex.

\section{Preliminaries}

In this section we are going to recall some concepts and known results concerning Coxeter and Artin groups, Artin monoids, the Salvetti complex, the $K(\pi,1)$ conjecture, and discrete Morse theory.

\subsection{Coxeter and Artin groups}
Let $S$ be a finite set, and let $M = (m_{s,t})_{s,t\in S}$ be a square matrix indexed by $S$ and satisfying the following properties:
\begin{itemize}
 \item $m_{s,t} \in \{2,3,\dots \}\cup\{\infty\}$ for all $s\neq t$, and $m_{s,t}=1$ for $s=t$;
 \item $m_{s,t}=m_{t,s}$, i.e.\ $M$ is symmetric.
\end{itemize}
Such a matrix is called a \emph{Coxeter matrix}.
From a Coxeter matrix $M$, we define the corresponding \emph{Coxeter group} as follows:
\[ W = \langle\, S \mid (st)^{m_{s,t}}=1 \;\; \forall\, s,t\in S \text{ such that } m_{s,t}\neq \infty \,\rangle. \]
Given a subset $T\subseteq S$, let $W_T$ be the subgroup of $W$ generated by $T$. This is also a Coxeter group, with the natural structure deriving from the Coxeter matrix $M|_{T\times T}$ \cite{bourbaki1968elements}.
Also, define $\sf$ as the set of subsets $T$ of $S$ such that $W_T$ is finite.

Coxeter groups are naturally endowed with a length function $\ell\colon W \to \N$, namely the function that maps an element $w\in W$ to the minimal length of an expression of $w$ as a product of generators in $S$.

Since in $W$ all the generators of $S$ have order 2, the relations $(st)^{m_{s,t}}=1$ for $s\neq t$ can be also written as
\[ \pr{s}{t}{m_{s,t}} = \pr{t}{s}{m_{s,t}}, \]
where the notation $\pr{a}{b}{m}$ stands for
\[ \pr{a}{b}{m} = \begin{cases}
		    (ab)^{\frac m2} & \text{if $m$ is even}, \\
		    (ab)^{\frac{m-1}{2}}a & \text{if $m$ is odd}.
                  \end{cases} \]
For instance, if $m_{s,t}=3$ the relation $(st)^3=1$ can be written as $sts=tst$.
So we have that
\[ W = \left\langle\, S \;\left\lvert\;
  \begin{array}{ll}
    s^2=1 & \forall\, s\in S, \\
    \pr{s}{t}{m_{s,t}} = \pr{t}{s}{m_{s,t}} & \forall \,\, s,t\in S \text{ such that } m_{s,t}\neq \infty
  \end{array}
  \right. \right\rangle. \]
Consider now the set $\Sigma = \{ \sigma_s \mid s\in S \}$, which is in natural bijection with $S$, and define an \emph{Artin group} as follows.
\[ A = \langle\, \Sigma \mid
  \pr{\sigma_s}{\sigma_t}{m_{s,t}} = \pr{\sigma_t}{\sigma_s}{m_{s,t}} \; \forall \,\, s,t\in S \text{ such that } m_{s,t}\neq \infty
  \rangle. \]
Clearly, for any Coxeter matrix $M$ there is a natural projection $\pi\colon A \to W$ between the corresponding Artin and Coxeter groups, sending $\sigma_s$ to $s$ for all $s\in S$.

If $W$ is finite, then we say that $A$ is an Artin group \emph{of finite type}.

\subsection{Artin monoids}
The \emph{Artin monoid} corresponding to a Coxeter matrix $M$ is the monoid given by
  \[ A^+ = \langle\, \Sigma \mid
  \pr{\sigma_s}{\sigma_t}{m_{s,t}} = \pr{\sigma_t}{\sigma_s}{m_{s,t}} \; \forall \,\, s,t\in S \text{ such that } m_{s,t}\neq \infty
  \rangle, \]
where this is a monoid presentation (as opposed to the group presentation for $A$).
The reason why we take the freedom to use the same generating set $\Sigma$ for $A^+$ and for $A$ is given by the following theorem.

\begin{theorem}[\cite{paris2002artin}]
  The natural monoid homomorphism $A^+ \to A$ is injective.%
  \label{thm:artin-monoid-injects-in-group}
\end{theorem}

In view of Theorem \ref{thm:artin-monoid-injects-in-group}, from now on we will consider $A^+$ as contained in $A$.
The Artin monoid is also called \emph{positive monoid} of $A$, for its elements are precisely those which can be written as a product (with positive exponents) of generators in $\Sigma$.
An immediate consequence of the previous theorem is that the Artin monoid is left and right cancellative (this was originally proved in \cite{brieskorn1972artin}).

Since the relations $\pr{\sigma_s}{\sigma_t}{m_{s,t}}=\pr{\sigma_t}{\sigma_s}{m_{s,t}}$ involve the same number of generators on the left hand side and on the right hand side, there is a well-defined length function $\ell\colon A^+ \to \N$ that sends an element $\sigma_{s_1}\cdots \sigma_{s_k} \in A^+$ to the length $k$ of its representation.

\begin{definition}
  Given $\alpha,\beta\in A^+$, we say that $\alpha \leql \beta$ if there exists $\gamma\in A^+$ such that $\alpha\gamma=\beta$.
  Similarly we say that $\alpha \leqr\beta$ if there exists $\gamma\in A^+$ such that $\gamma\alpha = \beta$.
\end{definition}

If $\alpha\leql\beta$ we also say that $\alpha$ is a left divisor of $\beta$, that $\alpha$ left divides $\beta$, or that $\beta$ is left divisible by $\alpha$. We do the same for right divisibility.
Both $\leql$ and $\leqr$ are partial order relations on $A^+$ \cite{brieskorn1972artin}.

\begin{definition}
  Let $E$ be a subset of $A^+$.
  A \emph{left common divisor} of $E$ is any element of $A^+$ which left divides all elements of $E$. A \emph{greatest left common divisor} of $E$ is a left common divisor of $E$ which is left multiple of all the left common divisors of $E$.
  Similarly, a \emph{left common multiple} of $E$ is any element of $A^+$ which is left multiple of all elements of $E$, and a \emph{least left common multiple} of $E$ is a left common multiple of $E$ which a left divisor of all the left common multiples of $E$.
  Define in the obvious way the analogous concepts for right divisibility.
\end{definition}

\begin{proposition}[\cite{brieskorn1972artin}]
  When a greatest common divisor or a least common multiple exists for a set $E$, then it is unique.
\end{proposition}

\begin{proposition}[\cite{brieskorn1972artin}]
  Let $E$ be a subset of $A^+$.
  If $E$ admits a left (resp.\ right) common multiple, then it also admits a least left (resp.\ right) common multiple.
  \label{prop:existence-lcm}
\end{proposition}

\begin{proposition}[\cite{brieskorn1972artin}]
  Any non-empty subset $E$ of $A^+$ admits a greatest left common divisor and a greatest right common divisor.
\end{proposition}

We are now going to introduce the fundamental element of the Artin monoid, which is (when it exists) significantly important.
Recall that $\Sigma=\{\sigma_s\mid s\in S\}$.

\begin{theorem}[\cite{brieskorn1972artin}]
  For an Artin monoid $A^+$, the following conditions are equivalent:
  \begin{itemize}
   \item $A$ is of finite type;
   \item $\Sigma$ admits a least left common multiple;
   \item $\Sigma$ admits a least right common multiple.
  \end{itemize}
  Moreover, if they are satisfied, then the least left common multiple and the least right common multiple of $\Sigma$ coincide.
  \label{thm:existence-fundamental-element}
\end{theorem}

\begin{definition}
  If $M$ is a Coxeter matrix corresponding to an Artin group of finite type, the least left (or right) common multiple of $\Sigma$ in $A^+$ is called \emph{fundamental element} of $A^+$ and is usually denoted by $\Delta$.
\end{definition}

The following theorem summarizes some of the properties of the fundamental element.
Before that, two more definitions are required.

\begin{definition}
  An element $\alpha\in A^+$ is \emph{square-free} if it cannot be written in the form $\beta \sigma_s^2\gamma$ for $\beta,\gamma\in A^+$ and $s\in S$.
\end{definition}

\begin{definition}
  Let $\rev\colon A^+ \to A^+$ be the bijection that sends an element $\sigma_{s_1}\sigma_{s_2}\cdots \sigma_{s_k}\in A^+$ to its ``reverse'' $\sigma_{s_k}\sigma_{s_{k-1}}\cdots \sigma_{s_1}$.
  It is easy to check that it is well-defined.
\end{definition}

\begin{theorem}[\cite{brieskorn1972artin}]
  Let $M$ be a Coxeter matrix corresponding to an Artin group of finite type, so that $A^+$ admits a fundamental element $\Delta$.
  Then the following properties hold:
  \begin{enumerate}[(i)]
    \item $\rev\Delta = \Delta$;
    \item an element of $A^+$ left divides $\Delta$ if and only if it right divides $\Delta$;
    \item an element of $A^+$ is square-free if and only if it is a (left or right) divisor of $\Delta$;
    \item the least (left or right) common multiple of square-free elements of $A^+$ is square-free;
    \item $\Delta$ is the uniquely determined square-free element of maximal length in $A^+$;
    \item any element $\alpha\in A$ can be written in the form $\alpha=\Delta^{-k}\beta$ for some $\beta \in A^+$ and $k\in\N$.
  \end{enumerate}
  \label{thm:properties-fundamental-element}
\end{theorem}

Consider now some subset $T\subseteq S$. Let $\Sigma_T = \{ \sigma_s\mid s\in T\}$ and let $A_T$ be the subgroup of $A$ generated by $\Sigma_T$.

\begin{theorem}[\cite{van1983homotopy}]
  The natural homomorphism $A_T \to A$ which sends $\sigma_s$ to $\sigma_s$ for all $s\in T$ is injective.
  In other words, $A_T$ is the Artin group corresponding to the Coxeter matrix $M|_{T\times T}$.
\end{theorem}

\begin{theorem}[\cite{brieskorn1972artin}]
  The least (left or right) common multiple of $\Sigma_T$ exists in $A^+$ if and only if the Coxeter matrix $M|_{T\times T}$ is of finite type.
  \label{thm:existence-fundamental-element-subgroups}
\end{theorem}

If $M|_{T\times T}$ is a Coxeter matrix of finite type, it makes sense to consider the fundamental element of the Artin monoid $A_T^+$ corresponding to $M|_{T\times T}$.
Such element will be denoted by $\Delta_T$.

\begin{lemma}[\cite{brieskorn1972artin}]
  $\Delta_T$ is precisely the least (left or right) common multiple of $\Sigma_T$ in $A^+$.
  \label{lemma:fundamental-element-as-common-multiple}
\end{lemma}

We will finally introduce a normal form for elements of the Artin monoid $A^+$. To do so, define for any $\alpha\in A^+$ the set
\[ I(\alpha) = \{ s\in S \mid \alpha=\beta\sigma_s \text{ for some } \beta \in A^+ \}. \]
In other words, $I(\alpha)$ is the set of elements $s\in S$ such that $\sigma_s$ right divides $\alpha$.

\begin{theorem}[\cite{brieskorn1972artin}]
  For any $\alpha\in A^+$ there exists a unique tuple $(T_1,\dots,T_k)$ of non-empty subsets of $S$ such that
  \[ \alpha = \Delta_{T_k}\Delta_{T_{k-1}}\cdots \Delta_{T_1} \]
  and $I(\Delta_{T_k}\cdots \Delta_{T_j})=T_j$ for $1\leq j\leq k$.
  \label{thm:normal-form-artin-monoid}
\end{theorem}

\subsection{The Salvetti complex and the \texorpdfstring{$K(\pi,1)$}{K(pi,1)} conjecture}

The Salvetti complex of a hyperplane arrangement was first introduced by Salvetti \cite{salvetti1987topology} for arrangements of affine hyperplanes, thus including Coxeter graphs of finite and affine type. Later it was generalized to the case of arrangements associated to arbitrary Coxeter graphs (see \cite{salvetti1994homotopy,paris2012k}).
We are going to define it as in \cite{paris2012k}, and we will quote some known results about it.

\begin{definition}
  Given a poset $(P,\leq)$, its \emph{derived complex} is a simplicial complex with $P$ as set of vertices and having a $k$-simplex $\{p_0,\dots,p_k\}$ for every chain $p_0 < p_1 < \ldots < p_k$ in $P$.
\end{definition}

\begin{definition}
  Let $T\subseteq S$.
  An element $w\in W$ is \emph{$T$-minimal} if it is the unique element of smallest length in the coset $wW_T$ (the uniqueness of such element is proved in \cite{bourbaki1968elements}).
\end{definition}

Consider now the set $W\times S^f$, with the following partial order: $(u,T)\leq (v,R)$ if $T\subseteq R$, $v^{-1}u\in W_R$ and $v^{-1}u$ is $T$-minimal (for a proof that this is indeed a partial order relation, see \cite{paris2012k}, Lemma 3.2).

\begin{lemma}[\cite{paris2012k}, Corollary 3.7 and Lemma 3.9]
  Let $(u,T)\in W\times \sf$, and set
  \begin{IEEEeqnarray*}{cCl}
    P &=&\{(v,R)\in W \times \sf \mid (v,R)\leq (u,T)\}, \\
    P_1 &=& \{(v,R)\in W \times \sf \mid (v,R)< (u,T)\}.
  \end{IEEEeqnarray*}%
  Call $P'$ and $P_1'$ the geometric realizations of the derived complexes of $(P,\leq)$ and $(P_1,\leq)$, respectively.
  Then the pair $(P',P_1')$ is homeomorphic to the pair $(D^n,S^{n-1})$ for $n=|T|$.
  \label{lemma:cells-salvetti-complex}
\end{lemma}

\begin{definition}
  The \emph{Salvetti complex} of a Coxeter matrix $M$, denoted by $\sal(M)$, is the geometric realization of the derived complex of $(W\times \sf,\leq)$.
  By Lemma \ref{lemma:cells-salvetti-complex} it has a CW structure with one cell $C(u,T)$ for all $(u,T)\in W\times \sf$, where the dimension of a cell $C(u,T)$ is $|T|$.
\end{definition}

The Coxeter group $W$ acts on $W \times \sf$ by left-multiplication on the first coordinate, and thus also acts on $\sal(M)$. Such action is free, properly discontinuous and cellular, so the quotient map
\[ \sal(M) \to \sal(M)/W \]
is a covering map. Moreover such covering map induces a CW structure on the quotient space $\bsal(M)=\sal(M)/W$.
The complex $\bsal(M)$ has one cell $\bar C(T)$ of dimension $|T|$ for each $T\in S^f$.

Following \cite{paris2012k}, let us describe in more detail the combinatorics of the low-dimensional cells of the complexes $\sal(M)$ and $\bsal(M)$.
\begin{itemize}
 \item The $0$-cells of $\sal(M)$ are in one-to-one correspondence with the elements of the Coxeter group $W$. For this reason we also denote a $0$-cell $C(w, \varnothing)$ simply by $w$.
 
 \item Since $\{s\}\in \sf$ for all $s\in S$, we have a 1-cell $C(w,\{s\})$ joining vertices $w$ and $ws$ for each $w\in W$ and $s\in S$.
 Notice that the 1-cell $C(ws,\{s\})$ joins vertices $w$ and $ws$, but is different from $C(w,\{s\})$.
 Orient the 1-cell $C(w,\{s\})$ from $w$ to $ws$.

 \item A 2-cell $C(w,\{s,t\})$ exists only if $\{s,t\}\in\sf$, i.e.\ if $m=m_{s,t}\neq \infty$.
 If it exists, such 2-cell is a $2m$-gon with vertices
 \begin{IEEEeqnarray*}{l}
   w, \, ws, \, wst, \, \dots, \, w\,\pr{s}{t}{m-1}, \, w\,\pr{s}{t}{m} = w\,\pr{t}{s}{m},\\
   w\,\pr{t}{s}{m-1},\, \dots,\, wt.
 \end{IEEEeqnarray*}
 See also Figure \ref{fig:salvetti-complex} for a representation of such cell in the case $m=3$.
 
 \begin{figure}[htbp]
  \begin{center}
  \includegraphics{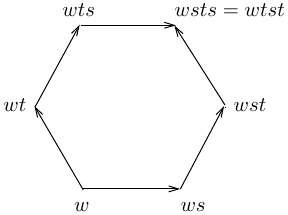}
  \end{center}
  \caption{Example of a $2$-cell $C(w,\{s,t\})$ of the complex $\sal(M)$, in the case $m_{s,t}=3$.}
  \label{fig:salvetti-complex}
 \end{figure}
\end{itemize}
The quotient complex $\bsal(M)$ has one $0$-cell $\bar C(\varnothing)$, a $1$-cell $\bar C(\{s\})$ for each $s\in S$, and a $2$-cell $\bar C(\{s,t\})$ for each $\{s,t\}\in \sf$.
Therefore the fundamental group of $\bsal(M)$ admits a presentation with a generator $\sigma_s$ for each $s\in S$ and a relation for each $2$-cell $\bar C(\{s,t\})$.
Such relation turns out to be exactly of the form
\[ \pr{\sigma_s}{\sigma_t}{m_{s,t}} = \pr{\sigma_t}{\sigma_s}{m_{s,t}}, \]
which means that the fundamental group of $\bsal(M)$ is naturally isomorphic to the corresponding Artin group $A$.

Given a Coxeter matrix $M$, a particular representation of the corresponding Coxeter group gives rise in a natural way to a certain topological space $N(M)$.
The $K(\pi,1)$ conjecture, due to Brieskorn \cite{brieskorn1973groupes} (for groups of finite type), Arnold, Pham, and Thom \cite{van1983homotopy} (in full generality), is the following.

\begin{conjecture}[$K(\pi,1)$ conjecture]
  The space $N(M)$ is a classifying space for the Artin group $A$ corresponding to $M$.
\end{conjecture}

We will not discuss further the definition of such space $N(M)$. We simply need to know that it is homotopy equivalent to the Salvetti complex $\bsal(M)$.
Therefore, the $K(\pi,1)$ conjecture is equivalent to the following.

\begin{conjecture}
  The CW complex $\bsal(M)$ is a classifying space for the corresponding Artin group $A$.
\end{conjecture}

The $K(\pi,1)$ conjecture would have important consequences in the theory of Artin groups. For instance, the existence of a finite CW model for the classifying space of $A$ implies that the homology of $A$ is finite-dimensional, and that $A$ is torsion-free (both these properties are not known in general).
The conjecture was proved only for some families of Artin groups; the most important result in this regard is probably the following.

\begin{theorem}[\cite{deligne1972immeubles}]
  The $K(\pi,1)$ conjecture holds for Artin groups of finite type.
  \label{thm:conjecture-finite-type}
\end{theorem}

\subsection{Discrete Morse theory}

Discrete Morse theory is a powerful tool for simplifying CW complexes while preserving their homotopy type.
It was first developed by Forman \cite{forman1998morse}, who presented it as a combinatorial analogue of Morse theory. Forman's version of discrete Morse theory, based on the concept of discrete Morse function, was later reformulated by Chari and Batzies in terms of acyclic matchings \cite{chari2000discrete,batzies2002discrete}.
Here we will briefly present the latter formulation, which we will use later.

Let $X$ be a CW complex.
Recall that each cell of $X$ has a characteristic map $\Phi\colon D^n\to X$ and an attaching map $\varphi\colon S^{n-1}\to X$, where $\varphi = \Phi|_{\partial D^n}$.

\begin{definition}
  The \emph{face poset} of $X$ is the set $\xstar$ of its open cells together with the partial order defined by $\sigma \leq \tau$ if $\bar\sigma \subseteq \bar\tau$.
\end{definition}

\begin{definition}[\cite{forman1998morse}]
  Let $\sigma,\tau\in \xstar$.
  If $\dim\tau= \dim\sigma+1$ and $\sigma \leq \tau$ we say that $\sigma$ is a \emph{face} of $\tau$.
  We say that $\sigma$ is a \emph{regular face} of $\tau$ if, in addition, the two following conditions hold (set $n=\dim\sigma$ and let $\Phi$ be the characteristic map of $\tau$):
  \begin{enumerate}[(i)]
    \item $\Phi|_{\Phi^{-1}(\sigma)}\colon \Phi^{-1}(\sigma) \to \sigma$ is a homeomorphism;
    \item $\smash{\overline{\Phi^{-1}(\sigma)}}$ is a closed metric $n$-ball in $D^{n+1}$.
  \end{enumerate}
\end{definition}

\begin{definition}
  $X$ is a \emph{regular CW complex} if all its faces are regular.
  \label{def:regular-cw-complex}
\end{definition}


\begin{definition}
  The \emph{cell graph} $G_X$ of $X$ is the Hasse diagram of $\xstar$, i.e.\ a directed graph with $\xstar$ as set of vertices and an edge from $\tau$ to $\sigma$ (written $\tau\to\sigma$) if $\sigma$ is a face of $\tau$.
  Denote the set of edges of $G_X$ by $E_X$.
  \label{def:complex-graph}
\end{definition}

\begin{definition}
  A \emph{matching} on $X$ is a subset $\M\subseteq E_X$ such that
  \begin{enumerate}[(i)]
    \item if $(\tau\to\sigma)\in \M$, then $\sigma$ is a regular face of $\tau$;
    \item any cell of $X$ occurs in at most one edge of $\M$.
  \end{enumerate}
\end{definition}

Given a matching $\M$ on $X$, define a graph $G_X^\M$ obtained from $G_X$ by inverting all the edges in $\M$.

\begin{definition}
  A matching $\M$ on $X$ is \emph{acyclic} if the corresponding graph $G_X^\M$ is acyclic.
\end{definition}

The aim of discrete Morse theory is to construct, from a CW complex $X$ with an acyclic matching $\M$, a simpler CW complex $X_\M$ (called \emph{Morse complex}) homotopy equivalent to $X$ but with fewer cells.

\begin{definition}
  Let $\M$ be an acyclic matching on $X$.
  A cell of $X$ is \emph{$\M$-essential} if it doesn't occur in any edge of $\M$.
\end{definition}

\begin{definition}
  Let $(P,\leq)$ be a poset. A \emph{$P$-grading} on $X$ is a poset map $\eta\colon \xstar\to P$.
  Given a $P$-grading on $X$, for any $p\in P$ denote by $X_{\leq p}$ the subcomplex of $X$ consisting of all the cells $\sigma$ such that $\eta(\sigma)\leq p$.
\end{definition}

\begin{definition}
  A $P$-grading on $X$ is \emph{compact} if $X_{\leq p}$ is compact for all $p\in P$.
\end{definition}

\begin{definition}
  Let $\M$ be an acyclic matching on $X$ and $\eta$ a $P$-grading on $X$.
  We say that $\M$ and $\eta$ are \emph{compatible} if $\eta(\sigma)=\eta(\tau)$ for all $(\tau\to\sigma)\in \M$.
  In other words, the matching $\M$ can be written as union of matchings $\M_p$ for $p\in P$, where each $\M_p$ is a matching on the fiber $\eta^{-1}(p)$.
\end{definition}

\begin{theorem}[\cite{batzies2002discrete}]
  Let $X$ be a CW complex with an acyclic matching $\M$ and a compact $P$-grading $\eta$, such that $\M$ and $\eta$ are compatible.
  \begin{enumerate}[(i)]
    \item There exist a CW complex $X_\M$ and a homotopy equivalence $f_\M \colon X \to X_\M$.
    \item The $n$-cells of $X_\M$ are in one-to-one correspondence with the $\M$-essential $n$-cells of $X$.
  \end{enumerate}
  In addition, the construction of the couple $(X_\M, f_\M)$ is natural with respect to inclusion: let $Y$ be a subcomplex of $X$ such that, if $(\tau\to\sigma)\in \M$ and $\sigma\in \ystar$, then $\tau\in \ystar$; then $Y_{\M'}$ is naturally included in $X_\M$ and the diagram
  \begin{center}
  \begin{tikzcd}
    Y \arrow{r}{f_{\M'}} \arrow[hookrightarrow]{d}{}
    & Y_{\M'} \arrow[hookrightarrow]{d}{} \\
    X \arrow{r}{f_{\M}}
    & X_{\M}
  \end{tikzcd}
  \end{center}
  is commutative, where $\M'$ is the restriction of $\M$ to $G_Y$.
  The CW complex $X_\M$ is called \emph{Morse complex} of $X$ with respect to the acyclic matching $\M$.
  \label{thm:discrete-morse-theory}
\end{theorem}

Let us finally present a lemma which is useful for constructing acyclic matchings.

\begin{lemma}
  Let $\M$ be a matching on $X$ and let $\eta$ be a $P$-grading on $X$ compatible with $\M$. Let $\M_p$ be the restriction of $\M$ to the fiber $\eta^{-1}(p)$, for all $p\in P$. If $\M_p$ is acyclic for all $p\in P$, then $\M$ is also acyclic.
  \label{lemma:acyclic-matching}
\end{lemma}

\begin{proof}
  Suppose by contradiction that the graph $G_X^\M$ contains a cycle. Since the edges in $\M$ increase the dimension by $1$ whereas all the others lower the dimension by $1$, a cycle must be of the form
  \[ \tau_1 \xrightarrow{\phantom{\M}} \sigma_1 \xrightarrow{\M} \tau_2 \xrightarrow{\phantom{\M}} \sigma_2 \xrightarrow{\M}
  \dots
  \xrightarrow{\phantom{\M}} \sigma_{k-1} \xrightarrow{\M} \tau_k \xrightarrow{\phantom{\M}} \sigma_k \xrightarrow{\M} \tau_1, \]
  where the edges labeled with $\M$ are exactly those belonging to $\M$.
  Since $\tau_i \geq \sigma_i$ in $\xstar$ we have that $\eta(\tau_i) \geq \eta(\sigma_i)$ in $P$, for all $i=1,\dots,k$.
  Moreover $\eta(\sigma_i) = \eta(\tau_{i+1})$ since $(\tau_{i+1}\to\sigma_i)\in \M$, for all $i=1,\dots,k$ (the indices are taken modulo $k$).
  Therefore
  \[ \eta(\tau_1) \geq \eta(\sigma_1) = \eta(\tau_2) \geq \eta(\sigma_2) = \dots = \eta(\tau_k) \geq \eta(\sigma_k) = \eta(\tau_1). \]
  The first and the last term of this chain of inequalities are equal, so all the terms are equal to the same element $p\in P$.
  Then this cycle is contained in the graph $G_X^{\M_p}$ and therefore $\M_p$ is not acyclic, which is a contradiction.
\end{proof}

In view of Lemma \ref{lemma:acyclic-matching}, it possible to weaken the hypothesis of Theorem \ref{thm:discrete-morse-theory} by removing the requirement of $\M$ being acyclic and asking instead that $\M_p$ is acyclic for all $p\in P$ (where $\M_p$ is the restriction of $\M$ to the fiber $\eta^{-1}(p)$).
In this way the $P$-grading $\eta$ is used to obtain both compactness and acyclicity.

\section{The classifying space of monoids}

We are now going to introduce the notion of classifying space of a monoid, as a particular case of the classifying space of a small category (viewing a monoid as a category with one object) \cite{segal1973configuration}.

\begin{definition}
  The \emph{classifying space} $BM$ of a monoid $M$ is the geometric realization of the following simplicial set. The $n$-simplices are given by the sequences $(x_1,\dots,x_n)$ of $n$ elements of $M$, denoted by the symbol $[x_1|\dots|x_n]$.
  The face maps send an $n$-simplex $[x_1|\dots|x_n]$ to the simplices $[x_2|\dots|x_n]$, $[x_1|\dots|x_ix_{i+1}|\dots|x_n]$ for $i=1,\dots,n-1$, and $[x_1|\dots|x_{n-1}]$.
  The degeneracy maps send $[x_1|\dots|x_n]$ to $[x_1|\dots|x_i|1|x_{i+1}|\dots |x_n]$ for $i=0,\dots,n$.
  \label{def:classifying-space}
\end{definition}

As shown in \cite{milnor1957geometric}, the geometric realization of a simplicial set is a CW complex having a $n$-cell for each non-degenerate $n$-simplex.
Therefore the classifying space of a monoid is a CW complex having as $n$-cells the simplices $[x_1|\dots|x_n]$ with $x_i\neq 1$ for all $i$.
Notice also that $BM$ has only one $0$-cell denoted by $[\,]$.

\begin{definition}
  The \emph{group of fractions} of a monoid $M$ is a group $G$ together with a homomorphism $M\to G$ satisfying the following universal property: for any group $H$ and homomorphism $M\to H$, there exists a unique homomorphism $G\to H$ which makes the following diagram commutative.
  \begin{center}
  \begin{tikzcd}
  M \arrow{r} \arrow{d} & G \arrow[dashed]{dl}\\
  H 
  \end{tikzcd}
  \end{center}
\end{definition}

\begin{remark}
  If any presentation of $M$ is given, then the group of fractions $G$ of $M$ is the group with the same presentation.
  \label{rmk:presentation-groupification}
\end{remark}

\begin{remark}
  The fundamental group of the classifying space $BM$ of a monoid $M$ is given by the group of fractions $G$ of $M$.
  This can be easily seen using the well-known presentation of the fundamental group of a CW complex with one $0$-cell: generators are given by the $1$-cells, and relations are given by the attaching maps of the $2$-cells.
  In our case the generator set is $\{[x]\mid x\in M, \, x\neq 1\}$ and the relation corresponding to the $2$-cell $[x|y]$ is given by $[x][y][xy]^{-1}$ if $xy\neq 1$ and $[x][y]$ if $xy=1$.
  This is indeed a presentation for the group of fractions $G$ of $M$, by Remark \ref{rmk:presentation-groupification}.
  \label{rmk:fundamental-group-BM}
\end{remark}

Before focusing on the case of Artin monoids, we are going to give an explicit construction for the universal cover of $BM$ for any monoid $M$ that injects in its group of fractions $G$ (i.e.\ when the natural map $M\to G$ is injective).
This construction generalizes the one of Example 1B.7 in \cite{hatcher}.
Let $EM$ be the geometric realization of the simplicial set whose $n$-simplices are given by the $(n+1)$-tuples $[g|x_1|\dots|x_n]$, where $g\in G$ and $x_i\in M$.
The $i$-th face map sends $[g|x_1|\dots|x_n]$ to
\[
  \begin{cases}
    [gx_1|x_2|\dots|x_n] & \text{for $i=0$}; \\
    [g|x_1|\dots|x_ix_{i+1}|\dots|x_n] & \text{for $1\leq i\leq n-1$}; \\
    [g|x_1|\dots|x_{n-1}] & \text{for $i=n$}.
  \end{cases}
\]
The $i$-th degeneracy map sends $[g|x_1|\dots|x_n]$ to $[g|x_1|\dots|x_i|1|x_{i+1}|\dots|x_n]$ for all $i=0,\dots,n$.
Notice that the vertices of $EM$ are in bijection with the group $G$, and that the vertices of an $n$-simplex $[g|x_1|\dots|x_n]$ are of the form $[gx_1\cdots x_i]$ for $i=0,\dots,n$.
The group $G$ acts freely and simplicially on $EM$ by left multiplication: an element $h\in G$ sends the simplex $[g|x_1|\dots|x_n]$ to the simplex $[hg|x_1\dots|x_n]$. Thus the quotient map $EM\to EM/G$ is a covering map.

\begin{lemma}
  $EM/G$ is naturally homeomorphic to $BM$.
  \label{lemma:group-action-on-EM}
\end{lemma}

\begin{proof}
  We will make use of the fact that taking the quotient by the action of $G$ commutes with geometric realization.
  A simplex $[x_1|\dots|x_n]$ of $BM$ can be identified with the equivalence class of the simplex $[1|x_1|\dots|x_n]$ of $EM/G$. This identification is bijective and respects face maps and degeneracy maps, so it is an isomorphism of simplicial sets.
  Therefore it yields a homeomorphism at the level of geometric realizations.
\end{proof}

\begin{proposition}
  The space $EM$ is the universal cover of $BM$, with the natural covering map $p\colon EM\to BM$ obtained composing the quotient map $EM\to EM/G$ and the homeomorphism $EM/G\to BM$ of Lemma \ref{lemma:group-action-on-EM}.
  \label{prop:EM-universal-cover-of-BM}
\end{proposition}

\begin{proof}
  We have already seen that $p$ is indeed a covering map.
  Therefore it is enough to show that $EM$ is simply connected.
  Choose $[\,]$ and $[1]$ as basepoints for $BM$ and $EM$, respectively, so that $p\colon (EM,*)\to (BM,*)$ becomes a basepoint-preserving covering map.
  An element $c$ of $\pi_1(BM,*)$ can be represented as a signed sequence $(\epsilon_1[x_1], \dots, \epsilon_k[x_k])$ of $1$-cells, where the sign $\epsilon_i = \pm 1$ indicates whether the arc $[x_i]$ is traveled in positive or negative direction.
  If we lift such path to the covering space $EM$, we obtain a path passing through the vertices $[1], [x_1^{\epsilon_1}], [x_1^{\epsilon_1} x_2^{\epsilon_2}], \dots, [x_1^{\epsilon_1}x_2^{\epsilon_2}\cdots x_k^{\epsilon_k}]$.
  Notice that under the isomorphism $\pi_1(BM,*) \cong G$ of Remark \ref{rmk:fundamental-group-BM} the path $c$ corresponds precisely to $x_1^{\epsilon_1}x_2^{\epsilon_2}\cdots x_k^{\epsilon_k}$.
  This means that if $c$ is non-trivial in $\pi_1(BM,*)$ then it lifts to a non-closed path in $EM$.
  Since $p_*\colon \pi_1(EM,*)\to \pi_1(BM,*)$ is injective, we can conclude that $\pi_1(EM,*)$ is trivial.
\end{proof}

The space $EM$ has a particular subcomplex $E^+M$ consisting of all the cells $[g|x_1|\dots|x_n]$ such that $g\in M$.
In analogy to the case when $M$ is a group (for which $EM=E^+M$), we prove that $E^+M$ is contractible.

\begin{proposition}
  The space $E^+M$ is contractible.
  \label{prop:positive-EM-contractible}
\end{proposition}

\begin{proof}
  Let $\mathcal{C}$ be the category with objects given by the elements of $M$ and morphisms $x\to y$ whenever $x \leql y$.
  Then $E^+M$ can be regarded as the geometric realization of the nerve of $\mathcal{C}$.
  Since $\mathcal{C}$ has an initial object, its nerve is contractible.
\end{proof}

\section{The Salvetti complex and Artin monoids}

In \cite{dobrinskaya2006configuration}, Dobrinskaya proved that the $K(\pi,1)$ conjecture can be reformulated as follows.

\begin{theorem}[\cite{dobrinskaya2006configuration}]
  The $K(\pi,1)$ conjecture holds for an Artin group $A$ if and only if the natural map $BA^+\to BA$ is a homotopy equivalence.
  \label{thm:dobrinskaya}
\end{theorem}

Dobrinskaya's theorem is particularly interesting since it relates the $K(\pi,1)$ conjecture to the problem of determining when the natural map $M\to G$ between a monoid and its group of fractions induces a homotopy equivalence $BM\to BG$ between the corresponding classifying spaces. Such phenomenon is known to happen in some cases (see \cite{mcduff1976homology}), but the general problem is still open.

To prove Theorem \ref{thm:dobrinskaya}, Dobrinskaya also proved the following result.

\begin{theorem}[\cite{dobrinskaya2006configuration}]
  The space $N(M)$ is homotopy equivalent to the classifying space $BA^+$ of the Artin monoid $A^+$ corresponding to $M$.
  \label{thm:dobrinskaya2}
\end{theorem}

It is quite easy to deduce Theorem \ref{thm:dobrinskaya} from Theorem \ref{thm:dobrinskaya2}.
Indeed, if the natural map $BA^+\to BA$ is a homotopy equivalence then
\[ N(M) \simeq BA^+ \simeq BA, \]
so the $K(\pi,1)$ conjecture holds for the Artin group $A$.
On the other hand, if the $K(\pi,1)$ conjecture holds for $A$ then both spaces $BA$ and $BA^+$ are classifying spaces for $A$, so the natural map $BA^+ \to BA$ must be a homotopy equivalence since it induces an isomorphism at the level of fundamental groups.

In the rest of this section we will present a new proof of Theorem \ref{thm:dobrinskaya2} based on discrete Morse theory.
Some ideas are taken from a recent work of Ozornova \cite{ozornova2017discrete}, but we will prove the stronger statement that the space $BA^+$ can be collapsed (in the sense of discrete Morse theory) to obtain a CW complex which is naturally homotopy equivalent to the Salvetti complex $\bsal(M)$.
This in particular answers some of the questions left open in \cite{ozornova2017discrete}, Section 7.

From now on, let $M$ be a Coxeter matrix and $A$ the corresponding Artin group.
When $A$ is of finite type we are able to show that $EA^+$ is contractible.
Before doing that, we recall the following classical result which be deduced from \cite{hatcher}, Corollary 4G.3.

\begin{lemma}
  Let $X$ be a CW complex, and let $\{Y_i \mid i\in\N \}$ be a family of contractible subcomplexes of $X$ such that $Y_i\subseteq Y_{i+1}$ for all $i\in \N$ and
  \[ \bigcup_{i\in\N} Y_i = X. \]
  Then $X$ is also contractible.
  \label{lemma:contractible-chain}
\end{lemma}

\begin{theorem}
  If $A$ is an Artin group of finite type, then the space $EA^+$ is contractible.
  \label{thm:EM-contractible-finite-type}
\end{theorem}

\begin{proof}
  To make the notation more readable denote the space $EA^+$ by $X$ and its subcomplex $E^+A^+$ by $X^+$.
  For any $h\in A$ the subcomplex $h X^+$ is homeomorphic to $X^+$, and therefore it is contractible by Proposition \ref{prop:positive-EM-contractible}.
  Notice also that $hX^+$ consists of all the simplices $[g|x_1|\dots|x_n]$ of $X$ such that $h\leql g$ (here we extended the partial order $\leql$ to $A$, so that $h \leql g$ means $h^{-1}g \in A^+$).
  Then by Theorem \ref{thm:properties-fundamental-element}, if $\Delta$ is the fundamental element of $A^+$, $X$ is the union of the subcomplexes $Y_i = \Delta^{-i}X^+$ for $i\in\N$. Since $Y_i\subseteq Y_{i+1}$ for all $i$, we can apply Lemma \ref{lemma:contractible-chain} to conclude that $X$ is contractible.
\end{proof}

\begin{corollary}
  If $A$ is an Artin group of finite type, then the classifying space $BA^+$ is a classifying space for $A$.
  \label{cor:finite-type-classifying-space}
\end{corollary}

\begin{proof}
  This result follows immediately by Remark \ref{rmk:fundamental-group-BM}, Proposition \ref{prop:EM-universal-cover-of-BM} and Theorem \ref{thm:EM-contractible-finite-type}.
\end{proof}

We are going to construct an acyclic matching $\M$ on $BA^+$ which is essentially a combination of the two matchings used in \cite{ozornova2017discrete}, with the difference that ours will be entirely on the topological level.
Set $Z = BA^+$ and
\[ \D = \{ \Delta_T \mid T\in \sf\setminus \{\varnothing\} \}, \]
where $\Delta_T$ is the fundamental element of $A_T^+\subseteq A^+$.
First we are going to describe some definitions and results of \cite{ozornova2017discrete}, which will lead to the construction of two matchings $\M_1$ and $\M_2$.
We start with a regularity criterion for faces in $Z$.

\begin{lemma}[\cite{ozornova2017discrete}]
  Call $\delta_i$ the $i$-th face map of $Z$.
  Let $c_1$ be a non-degenerate $n$-simplex of $Z$ and let $c_2=\delta_i(c_1)$ be a face of $c_1$.
  If $\delta_j(c_1) \neq c_2$ for all $j\neq i$, then $c_2$ is a regular face of $c_1$.
  \label{lemma:regularity-criterion}
\end{lemma}

To construct the first matching $\M_1$, we give the following definitions.
\begin{itemize}
  \item A cell $c = [x_1|\dots|x_n] \in \zstar$ is \emph{$\mu_1$-essential} if the products $x_k\cdots x_n$ lie in $\D$ for $1\leq k\leq n$.
  
  \item The \emph{$\mu_1$-depth} of a cell $c = [x_1|\dots|x_n]$ is given by
  \[ d_1(c) = \min \{ j \mid [x_j|\dots |x_n] \text{ is $\mu_1$-essential} \},  \]
  with the convention that $d_1(c)=n+1$ if no such $j$ exists.
  Notice that $c$ is $\mu_1$-essential if and only if $d_1(c)=1$.
  
  \item For a cell $c = [x_1|\dots|x_n]$ of $\mu_1$-depth $d$, and for $d\leq k\leq n+1$, define $I_k\subseteq S$ to be the unique subset of $S$ with the property that $x_k\cdots x_n = \Delta_{I_k}$.
  Notice that $I_{n+1} = \varnothing$.
  
  \item A cell $c = [x_1|\dots|x_n]$ of $\mu_1$-depth $d>1$ is \emph{$\mu_1$-collapsible} if
  \[ I(x_{d-1}x_d\cdots x_n) = I_d. \]
\end{itemize}

\begin{lemma}[\cite{ozornova2017discrete}]
  Define
  \[ \M_1 = \left\{\, (c_1 \to c_2) \;\left\lvert\;
  \begin{array}{l}
    c_1 = [x_1|\dots|x_n] \in \zstar \text{ is $\mu_1$-collapsible, and} \\
    c_2 = [x_1|\dots|x_{d-1}x_d|\dots|x_n] \text{ where } d = d_1(c_1)
  \end{array}
  \right. \right\}. \]
  Then $\M_1$ is an acyclic matching on $Z$ with essential cells given by the $\mu_1$-essential cells defined above.
\end{lemma}

To construct the second matching $\M_2$, assume from now on that the set $S$ carries a total order $\leq$.
Notice that a $\mu_1$-essential cell $c=[x_1|\dots|x_n]$ is completely characterized by the sequence of subsets $I_1 \supset I_2\supset \dots \supset I_n \supset I_{n+1} = \varnothing$ defined above.

\begin{itemize}
  \item A $\mu_1$-essential cell $c=[x_1|\dots|x_n]$ is \emph{$\mu_2$-essential} if, for any $1\leq k\leq n$, $I_k\setminus I_{k+1} = \{s_k\}$ and $s_k = \max I_k$.
  
  \item The \emph{$\mu_2$-depth} of a $\mu_1$-essential cell $c=[x_1|\dots|x_n]$ is given by
  \[ d_2(c) = \min \{ j \mid [x_j|\dots |x_n] \text{ is $\mu_2$-essential} \},  \]
  
  \item A $\mu_1$-essential cell $c = [x_1|\dots|x_n]$ of $\mu_2$-depth $d>1$ is \emph{$\mu_2$-collapsible} if
  \[ \max I_{d-1} = \max I_d. \]
\end{itemize}

\begin{lemma}[\cite{ozornova2017discrete}]
  Define
  \[ \M_2 = \left\{\, (c_1 \to c_2) \;\left\lvert\;
  \begin{array}{l}
    c_1 = [x_1|\dots|x_n] \in \zstar \text{ is $\mu_2$-collapsible, and} \\
    c_2 = [x_1|\dots|x_{d-1}x_d|\dots|x_n] \text{ where } d = d_2(c_1)
  \end{array}
  \right. \right\}. \]
  Then $\M_2$ is an acyclic matching on $Z$ with essential cells given by the non-$\mu_1$-essential cells and the $\mu_2$-essential cells.
\end{lemma}

\begin{proof}
  In \cite{ozornova2017discrete} this matching is constructed at the level of the algebraic complex.
  Therefore, to adapt the proof, we only need to check that $c_2$ is a regular face of $c_1$ whenever $(c_1 \to c_2)$ is in the matching.
  This follows directly from Lemma \ref{lemma:regularity-criterion}, since in $c_1$ we have that $x_i \neq 1$ for all $i$.
  Notice that the face regularity is the topological analogue of the $\Z$-compatibility proved in \cite{ozornova2017discrete}.
\end{proof}

Consider now the matching $\M=\M_1\cup \M_2$ on $Z$.
With a slight abuse of notation, define the length of a cell as
\[ \ell([x_1|\dots|x_n]) = \ell(x_1\cdots x_n) \]
for any cell $[x_1|\dots|x_n]\in \zstar$.
Define also a function $\eta\colon \zstar \to \N\times\{0,1\}$ as follows:
\[ \eta(c) = \begin{cases}
                             (\ell(c), 0) & \text{if $c$ is $\mu_1$-essential}; \\
                             (\ell(c), 1) & \text{if $c$ is not $\mu_1$-essential}.
                           \end{cases} \]

\begin{lemma}
  The function $\eta\colon \zstar \to \N\times\{0,1\}$ is a compact grading on $Z$, if $\N\times \{0,1\}$ is equipped with the lexicographic order.
\end{lemma}

\begin{proof}
  First we have to prove that $\eta$ is a poset map.
  For this it is enough to prove that, for any cell $c_1=[x_1|\dots|x_n] \in \zstar$ and for any cell $c_2$ which is a face of $c_1$, $\eta(c_1)\geq \eta(c_2)$.
  Suppose by contradiction that $\eta(c_1) < \eta(c_2)$ for some cells $c_1$ and $c_2$ as above.
  Since $\ell(c_1)\geq \ell(c_2)$ the only possibility is that $\eta(c_1)=(k,0)$ and $\eta(c_2)=(k,1)$ where $k=\ell(c_1)=\ell(c_2)$.
  This means in particular that $c_1$ is $\mu_1$-essential whereas $c_2$ is not.
  Since $\ell(c_1)=\ell(c_2)$, the cell $c_2$ must be of the form
  \[ c_2 = [x_1|\dots|x_ix_{i+1}|\dots|x_n] \]
  for some $i\in\{1,\dots,n-1\}$.
  Clearly if $c_1$ is $\mu_1$-essential then also $c_2$ is, so we obtain a contradiction.
  
  It only remains to prove that $Z_{(n,q)}$ is compact for all $(n,q)\in\N\times\{0,1\}$.
  This is immediate since $Z_{(n,q)}$ contains only cells of length $\leq n$ and there is only a finite number of them.
\end{proof}

\begin{proposition}
  The matching $\M$ on $Z$ is acyclic and compatible with the compact grading $\eta$.
\end{proposition}

\begin{proof}
  First let us prove that $\M$ and $\eta$ are compatible.
  If $(c_1\to c_2) \in \M_1$ then, by definition of $\M_1$, we have that $\ell(c_1)=\ell(c_2)$ and that both $c_1$ and $c_2$ are not $\mu_1$-essential.
  On the other hand, if $(c_1\to c_2)\in \M_2$, then $\ell(c_1)=\ell(c_2)$ and both $c_1$ and $c_2$ are $\mu_1$-essential.
  In any case we have $\eta(c_1)=\eta(c_2)$, which means that $\M$ and $\eta$ are compatible.
  
  Consider a fiber $\eta^{-1}(n,q)$, for some $(n,q)\in \N\times\{0,1\}$. It cannot simultaneously contain edges in $\M_1$ and edges in $\M_2$, because the value of $q$ determines whether the cells in $\eta^{-1}(n,q)$ must be $\mu_1$-essential or not.
  Since $\M_1$ and $\M_2$ are acyclic, the restriction of $\M$ to $\eta^{-1}(n,q)$ is also acyclic.
  This is true for all fibers $\eta^{-1}(n,q)$, therefore by Lemma \ref{lemma:acyclic-matching} we can conclude that $\M$ is also acyclic.
\end{proof}

The previous proposition allows to apply Theorem \ref{thm:discrete-morse-theory} to $Z$, obtaining a smaller CW complex which we will call $Y$.
The essential cells of the matching $\M$ are precisely the $\mu_2$-essential cells.
Notice that a $\mu_2$-essential cell $c = [x_1|\dots|x_n]\in\zstar$ is uniquely identified by the set $I(x_1\cdots x_n)\in \sf$.
This means that the cells of $Y$ are in one-to-one correspondence with $S^f$.

Call $e_T$ the cell of $Y$ corresponding to the set $T\in S^f$. Then $\dim e_T = |T|$.
By construction, every oriented path in the graph $G_Z^\M$ starting from a cell $[x_1|\dots|x_n]$ with all $x_i\in A_T^+$ ends in a cell $[x_1'|\dots|x_m']$ also satisfying $x_j'\in A_T^+$ for all $j$.
Therefore the attaching map of $e_T$ has image contained in the union of the cells $e_R$ with $R\subsetneq T$.
Thus any subset $\F\subseteq S^f$ which is closed under inclusion (i.e.\ $T\in \F$ and $R\subseteq T$ imply $R\in\F$) corresponds to a subcomplex $Y_\F$ of $Y$.
In particular this holds when $\F = T^f$ for any $T\subseteq S$.

In a similar way we have subcomplexes $\bsal_\F(M)$ of $\bsal(M)$ for all subsets $\F$ of $S^f$ closed under inclusion.

\begin{remark}
  The reduced complex $Y$ is natural with respect to inclusion of the set $S$.
  Indeed, consider Coxeter matrices $M$, $M'$ on the generating sets $S$, $S'$ such that $S\subseteq S'$ (with a fixed total order on $S'$, which induces a total order on $S$) and $M'|_{S\times S} = M$. Then we obtain reduced complexes $Y$ and $Y'$ such that $Y$ can be naturally identified with the subcomplex $\smash{Y'_{S^f}}$ of $Y'$.
  This is true because for any non-degenerate simplex $e$ of the subcomplex $BA_{M}^+\subseteq BA_{M'}^+$ all the faces of $e$, as well as the cell possibly matched with $e$, also belong to the subcomplex $BA_{M}^+$.
  Here we indicated with $A_M$ and $A_{M'}$ the Artin groups corresponding to $M$ and $M'$, respectively.
  \label{rmk:naturality-reduced-complex}
\end{remark}

Let us recall a few more results of homotopy theory which will be used later.
%
%
\begin{lemma}[\cite{hatcher}, Proposition 0.18]
  If $(X_1,A)$ is a CW pair and we have attaching maps $f,g\colon A\to X_0$ that are homotopic, then $X_0 \sqcup_f X_1 \simeq X_0 \sqcup_g X_1$ rel $X_0$.
\end{lemma}

\begin{corollary}
  If $X$ is a CW complex and $f,g\colon S^{n-1}\to X$ are two attaching maps of an $n$-cell $e^n$ that are homotopic, then $X\sqcup_f e^n \simeq X \sqcup_g e^n$ rel $X$.
  \label{cor:homotopic-attaching-maps}
\end{corollary}

\begin{proof}
  It follows from the previous lemma with $(X_1,A) = (D^n, S^{n-1})$ and $X_0=X$.
\end{proof}
%
%

We are finally ready to prove that the CW complexes $Y$ and $\bsal(M)$ are homotopy equivalent.
As a preliminary step, we give an explicit description of the attaching maps of the $2$-cells in $Y$.

\begin{lemma}
  Up to orientation, the boundary curve of a $2$-cell $e_{\{s,t\}}$ of $Y$ is given by
  \[
    \begin{cases}
      \pr{e_{\{s\}}}{e_{\{t\}}}{m_{s,t}} \;\pr{e_{\{s\}}^{-1}}{e_{\{t\}}^{-1}}{m_{s,t}} & \text{if $m_{s,t}$ is even}; \\
      \pr{e_{\{s\}}}{e_{\{t\}}}{m_{s,t}} \;\pr{e_{\{t\}}^{-1}}{e_{\{s\}}^{-1}}{m_{s,t}} & \text{if $m_{s,t}$ is odd}.
    \end{cases}
  \]
  \label{lemma:boundary-2-cell}
\end{lemma}

\begin{proof}
  By Remark \ref{rmk:naturality-reduced-complex} it is sufficient to treat the case $S=\{s,t\}$, so that $Y$ consists only of one $0$-cell $\smash{e_\varnothing}$, two $1$-cells $\smash{e_{\{s\}},e_{\{t\}}}$ and one $2$-cell $\smash{e_{\{s,t\}}}$.
  Suppose that $s>t$ (in the other case the result is the same but with reversed orientation).
  Set
  \[  \prp{a}{b}{k} = \begin{cases}
                   \pr{a}{b}{k} & \text{if $k$ is even}, \\
                   \pr{b}{a}{k} & \text{if $k$ is odd}. \\
                 \end{cases} \\ \]
  Moreover, set
  \begin{IEEEeqnarray*}{rCl}
    \product_s^k &=& \prp{\sigma_t}{\sigma_s}{k}, \\
    \product_t^k &=& \prp{\sigma_s}{\sigma_t}{k}. \\
  \end{IEEEeqnarray*}
  Essentially $\product_s^k$ is a product of $k$ alternating elements ($\sigma_s$ or $\sigma_t$) ending with $\sigma_s$, and $\product_t^k$ is the same but ending with $\sigma_t$. For example, $\product_s^4 = \sigma_t\sigma_s\sigma_t\sigma_s$. Set also $m=m_{s,t}$. The cells $\smash{e_\varnothing}$, $\smash{e_{\{s\}}}$, $\smash{e_{\{t\}}}$ and $\smash{e_{\{s,t\}}}$ correspond to the $\M$-essential cells of $Z$, which are the following:
  \begin{IEEEeqnarray*}{lCl}
    c_\varnothing &=& [\,], \\
    c_{\{s\}} &=& [\sigma_s], \\
    c_{\{t\}} &=& [\sigma_t], \\
    c_{\{s,t\}} &=& [\product_s^{m-1}|\sigma_t].
  \end{IEEEeqnarray*}
  
  If $c$ is a cell of $Z$, call $\M^{<c}$ the matching on $Z$ given by
  \[ \M^{<c} = \{ (c_1 \to c_2) \in \M \mid c_1 < c \text{ and } c_2 < c \}, \]
  where the inequalities on the right hand side are meant in the partial order induced by the acyclic graph $G_Z^\M$.
  Since $\M^{<c}$ is also an acyclic matching on $Z$, compatible with the compact grading $\eta$, we can consider the complex $Y^{<c}$ obtained collapsing $Z$ along the matching $\M^{<c}$.
  For simplicity, we will call a cell of some $Y^{<c}$ with the same name as the corresponding $\M^{<c}$-essential cell in $Z$.
  
  We want to prove by induction on $k$ the following two assertions:
  \begin{enumerate}[(i)]
    \item the boundary curve of the 2-cell $c=[\product_t^k|\sigma_s]$ in $Y^{<c}$ is
    \[ \prp{[\sigma_t]}{[\sigma_s]}{k+1}\, [\product_s^{k+1}]^{-1} \]
    for $1\leq k \leq m-1$;
    \item the boundary curve of the 2-cell $c=[\product_s^k|\sigma_t]$ in $Y^{<c}$ is
    \[ \prp{[\sigma_s]}{[\sigma_t]}{k+1}\, [\product_t^{k+1}]^{-1} \]
    for $1\leq k \leq m-2$.
  \end{enumerate}
  
  The base steps are for $k=1$. Start from case (i). We have $c=[\sigma_t|\sigma_s]$, whose boundary in $Z$ is given by $[\sigma_t][\sigma_s][\sigma_t\sigma_s]^{-1}$. The 1-cells $[\sigma_t]$ and $[\sigma_s]$ are $\M$-essential, whereas the 1-cell $[\sigma_t\sigma_s]$ is matched with $c$. Therefore all these 1-cells are $\M^{<c}$-essential.
  This means that the boundary of $c$ in $Y^{<c}$ is also given by $[\sigma_t][\sigma_s][\sigma_t\sigma_s]^{-1}$.
  Case (ii) is similar: we have $c=[\sigma_s|\sigma_t]$, and its boundary in $Z$ is $[\sigma_s][\sigma_t][\sigma_s\sigma_t]^{-1}$. This is also the boundary in $Y^{<c}$, because $[\sigma_s\sigma_t]$ is matched with $c$ (this wouldn't be true for $m=2$, but such case is excluded by the condition $k\leq m-2$).
  
  We want now to prove case (i) for $2\leq k \leq m-1$.
  We have $c=[\product_t^k|\sigma_s]$, whose boundary in $Z$ is given by $[\product_t^k][\sigma_s][\product_s^{k+1}]^{-1}$.
  The 1-cell $[\sigma_s]$ is $\M$-essential, so it is in particular $\M^{<c}$-essential.
  The 1-cell $[\product_s^{k+1}]$ is matched with $c$ in $\M$, so it is $\M^{<c}$-essential.
  Finally the 1-cell $[\product_t^k]$ is matched with $c'=[\product_s^{k-1}|\sigma_t]$, whose boundary in $Y^{<c'}$ is $\prp{[\sigma_s]}{[\sigma_t]}{k}\, [\product_t^{k}]^{-1}$ by induction.
  Since $c' < c$ in the partial order induced by $G_Z^\M$, the boundary of $c$ in $Y^{<c}$ is given by
  \[ \prp{[\sigma_s]}{[\sigma_t]}{k} \, [\sigma_s] \, [\product_s^{k+1}]^{-1} = \prp{[\sigma_t]}{[\sigma_s]}{k+1} \, [\product_s^{k+1}]^{-1}. \]
  See Figure \ref{fig:lemma-general-collapse} for a picture of general cell attachments after a Morse collapse on the $2$-skeleton.
  See the left part of Figure \ref{fig:lemma} for a picture of case (i).
  
  \begin{figure}
    \begin{center}
      \includegraphics{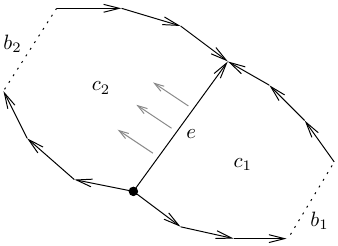}
    \end{center}
    \caption{Morse collapse on the $2$-skeleton.
    If two $2$-cells $c_1$ and $c_2$ have boundary curves $b_1e^{-1}$ and $b_2e^{-1}$ (where $e$ is a common $1$-cell, whereas $b_1$ and $b_2$ are sequences of $1$-cells), then collapsing $e$ with $c_2$ leaves a bigger cell $c_1$ with boundary curve $b_1b_2^{-1}$.}
    \label{fig:lemma-general-collapse}
  \end{figure}
  
  \begin{figure}
    \begin{center}
      \includegraphics{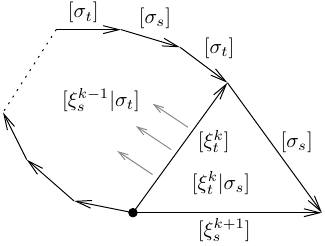} \qquad
      \includegraphics{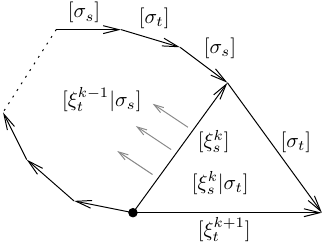}
    \end{center}
    \caption{On the left: the induction step, case (i); on the right: the induction step, case (ii). The boundary curve for all the 2-cells is denoted clockwise, starting from the black vertex.
    The light arrows indicate the Morse collapse.}
    \label{fig:lemma}
  \end{figure}

  We finally want to prove case (ii) for $2\leq k \leq m-2$.
  We have $c=[\product_s^k|\sigma_t]$, whose boundary in $Z$ is given by $[\product_s^k][\sigma_t][\product_t^{k+1}]^{-1}$.
  The 1-cell $[\sigma_t]$ is $\M$-essential, so it is $\M^{<c}$-essential.
  The 1-cell $[\product_t^{k+1}]$ is matched with $c$ in $\M$ (notice that this is true only for $k<m-1$, because $\product_t^m = \product_s^m$), thus it is also $\M^{<c}$-essential.
  The only 1-cell left to analyze is $[\product_s^k]$, which is matched with $c'=[\product_t^{k-1}|\sigma_s]$, whose boundary in $Y^{<c'}$ is $\prp{[\sigma_t]}{[\sigma_s]}{k}\, [\product_s^{k}]^{-1}$ by induction. 
  Then the boundary of $c$ in $Y^{<c}$ is given by
  \[ \prp{[\sigma_t]}{[\sigma_s]}{k} \, [\sigma_t] \, [\product_t^{k+1}]^{-1} = \prp{[\sigma_s]}{[\sigma_t]}{k+1} \, [\product_t^{k+1}]^{-1}. \]
  See the right part of Figure \ref{fig:lemma} for a picture of case (ii).
  
  The induction argument is complete. To end the proof, consider now the $2$-cell $c_{\{s,t\}} = [\product_s^{m-1}|\sigma_t]$ of $Z$. Its boundary is $[\product_s^{m-1}][\sigma_t][\product_t^m]^{-1}$.
  The $1$-cell $[\sigma_t]$ is $\M$-essential.
  The $1$-cell $[\product_s^{m-1}]$ is matched with $c' = [\product_t^{m-2}|\sigma_s]$, whose boundary in $Y^{<c'}$ is $\prp{[\sigma_t]}{[\sigma_s]}{m-1}\, [\product_s^{m-1}]^{-1}$.
  The $1$-cell $[\product_t^m] = [\product_s^m]$ is matched with $c'' = [\product_t^{m-1}|\sigma_s]$, whose boundary in $Y^{<c''}$ is $\prp{[\sigma_t]}{[\sigma_s]}{m}\, [\product_s^{m}]^{-1}$.
  Therefore the boundary of $c_{\{s,t\}}$ in $Y$ is
  \begin{IEEEeqnarray*}{rCl}
    && \prp{[\sigma_t]}{[\sigma_s]}{m-1}\; [\sigma_t] \; \Big( \prp{[\sigma_t]}{[\sigma_s]}{m} \Big)^{-1} \\
    &=& \prp{[\sigma_t]}{[\sigma_s]}{m-1}\; [\sigma_t] \; \pr{[\sigma_s]^{-1}}{[\sigma_t]^{-1}}{m} \\
    &=& \prp{[\sigma_s]}{[\sigma_t]}{m} \; \pr{[\sigma_s]^{-1}}{[\sigma_t]^{-1}}{m}.
  \end{IEEEeqnarray*}
  If $m$ is even we are done.
  If $m$ is odd, reversing the orientation yields
  \[ \Big( \prp{[\sigma_s]}{[\sigma_t]}{m} \; \pr{[\sigma_s]^{-1}}{[\sigma_t]^{-1}}{m} \Big)^{-1} = \,
     \pr{[\sigma_s]}{[\sigma_t]}{m} \; \pr{[\sigma_t]^{-1}}{[\sigma_s]^{-1}}{m}, \]
  so we are also done.
\end{proof}

\begin{theorem}
  For any Coxeter matrix $M$ there exists a homotopy equivalence $\psi \colon Y \to \bsal(M)$ such that, for every subset $\F$ of $S^f$ closed under inclusion, the restriction $\psi|_{Y_\F}$ has image contained in $\bsal_\F(M)$ and
  \[ \psi|_{Y_\F} \colon Y_\F \to \bsal_\F(M) \]
  is also a homotopy equivalence.
  \label{thm:equivalence-salvetti-morse}
\end{theorem}

\begin{proof}  
  We are going to construct simultaneously the map $\psi$ of the statement and its homotopy inverse $\psi' \colon \bsal(M) \to Y$, together with homotopies $F\colon Y\times [0,1] \to Y$ between $\psi' \circ \psi$ and $\id_Y$, and $F'\colon \bsal(M)\times [0,1] \to \bsal(M)$ between $\psi \circ \psi'$ and $\smash{\id_{\bsal(M)}}$.
  
  Consider a chain $\{\varnothing\} = \F_1\subseteq \F_2\subseteq \dots \subseteq \F_k = S^f$ of subsets of $S^f$ closed under inclusion and such that $|\F_{i+1}| = |\F_i| + 1$ for all $i$.
  We will define $\psi$, $\psi'$, $F$ and $F'$ recursively on the subcomplexes $Y_{\F_i}$ and $\bsal_{\F_i}(M)$, starting with the subcomplexes $Y_{\F_1}$ and $\bsal_{\F_1}$ consisting only of one $0$-cell, and extending them one cell at a time.
  We will construct the maps in such a way that $\psi|_{Y_{\F_i}}$ has image contained in $\smash{\bsal_{\F_i}(M)}$, $\smash{\psi'|_{\bsal_{\F_i}(M)}}$ has image contained in $Y_{\F_i}$, and the same will hold for the homotopies $F$ and $F'$ for any time $t\in [0,1]$.
  Simultaneously we will prove by induction that, for any subset $\F\subseteq\F_i$ closed under inclusion,
  \begin{itemize}
   \item the constructed maps
    \[ \psi|_{Y_{\F_i}}\colon Y_{\F_i}\to \bsal_{\F_i}(M) \quad \text{and} \quad
    \psi'|_{\bsal_{\F_i}(M)}\colon \bsal_{\F_i}(M)\to Y_{\F_i}, \]
    restricted to $Y_\F$ and $\bsal_\F(M)$, have image contained in $\bsal_\F(M)$ and $Y_\F$, respectively;
   \item the constructed homotopies
    \[ F|_{Y_{\F_i}\times [0,1]} \colon Y_{\F_i}\times [0,1] \to Y_{\F_i} \quad \text{and} \quad F'|_{\bsal_{\F_i}(M) \times [0,1]} \colon \bsal_{\F_i} \times [0,1] \to \bsal_{\F_i}, \]
    restricted to $Y_\F \times [0,1]$ and $\bsal_\F(M)\times [0,1]$, have image contained in $Y_\F$ and $\bsal_\F(M)$, respectively.
  \end{itemize}
  In particular this means that the restrictions of $\smash{\psi|_{Y_{\F_i}}}$ and $\smash{\psi'|_{\bsal_{\F_i}(M)}}$ to $Y_\F$ and $\bsal_\F(M)$ are homotopy inverses to each other through the homotopies obtained by restricting $\smash{F|_{Y_{\F_i}\times [0,1]}}$ and $\smash{F'|_{\bsal_{\F_i}(M) \times [0,1]}}$ to the subspaces $Y_\F \times [0,1]$ and $\bsal_\F(M)\times [0,1]$, respectively.
  
  Define $\psi|_{Y_{\F_1}}$ sending the $0$-cell of $Y$ to the $0$-cell of $\bsal(M)$, and $\smash{\psi'|_{\bsal_{\F_1}(M)}}$ as its inverse. Moreover define the homotopies $\smash{F|_{Y_{\F_1}\times [0,1]}}$ and $\smash{F'|_{\bsal_{\F_1}(M) \times [0,1]}}$ as the constant maps.
  
  Assume now by induction to have already defined the maps $\smash{\psi|_{Y_{\F_i}}}, \smash{\psi'|_{\bsal_{\F_i}(M)}}$, $F|_{Y_{\F_i}\times [0,1]}$ and $F'|_{\bsal_{\F_i}(M) \times [0,1]}$ for some $i$.
  To simplify the notation, set:
  \begin{IEEEeqnarray*}{CCLCCCL}
   X &=& Y_{\F_{i+1}}, & \qquad & X' &=& \bsal_{\F_{i+1}}(M), \\
   A &=& Y_{\F_i}, & \qquad & A' &=& \bsal_{\F_i}(M), \\
   \vartheta &=& \psi|_{A}, & \qquad & \vartheta' &=& \psi|_{A'}, \\
   G &=& F|_{A\times [0,1]}, & \qquad & G' &=& F'|_{A' \times [0,1]}.
  \end{IEEEeqnarray*}
  Let $T$ be the only element of $\sf$ which belongs to $\F_{i+1}$ but not to $\F_i$.
  Moreover let $e^n$ and $f^n$ be the $n$-cells corresponding to $T$ in $X$ and $X'$, respectively.
  We want to extend $\vartheta$ to $e^n$, $G$ to $e^n\times [0,1]$, $\vartheta'$ to $f^n$ and $G'$ to $f^n\times [0,1]$.
  
  If $n=1$ we simply send homeomorphically $e^1$ to $f^1$, preserving the orientation, and $f^1$ to $e^1$ with the inverse homeomorphism; the homotopies $G$ and $G'$ are then extended being constant on the new cells.
  If $n=2$ we can apply Lemma \ref{lemma:boundary-2-cell} to observe that the boundary curve of $e^2$ in $X$ is the same (via $\vartheta$) as the boundary curve of $f^2$ in $X'$; then we extend $\vartheta$ sending $e^2$ to $f^2$ homeomorphically, preserving the boundary, and similarly we extend $\vartheta'$.
  The homotopies $G$ and $G'$ can be extended to the new cells by Corollary \ref{cor:homotopic-attaching-maps}.

  Now we are going to deal with the case $n\geq 3$.
  Consider the following subsets of $S^f$, which are closed under inclusion:
  \[ \F = \{ R\in S^f \mid R\subsetneq T\}, \quad \F^* = \F \cup \{T\}. \]
  Notice that $\F\subseteq \F_i$ because $\F_{i+1}$ is closed under inclusion and $T$ is the only element in $\F_{i+1}\setminus\F_i$.
  The sets $\F$, $\F^*$, $\F_i$, $\F_{i+1}$ are related as follows.
  \begin{center}
    \begin{tikzcd}[column sep=large]
    \F_i \arrow[hookrightarrow]{r}{\text{add $T$}}
    & \F_{i+1} \\
    \F \arrow[hookrightarrow]{r}{\text{add $T$}} \arrow[hookrightarrow]{u}
    & \F^* \arrow[hookrightarrow]{u}
    \end{tikzcd}
  \end{center}
  Set $\h A = Y_{\F}$, $\h X = Y_{\F^*}$, $\h A' = \bsal_{\F}(M)$ and $\h X' = \bsal_{\F^*}(M)$.
  Then we have the following inclusions of CW complexes.
  \begin{center}
    \begin{tikzcd}[column sep=1.8cm]
    A \arrow[hookrightarrow]{r}{\text{attach $e^n$}}
    & X \\
    \widehat{A} \arrow[hookrightarrow]{r}{\text{attach $e^n$}} \arrow[hookrightarrow]{u}
    & \widehat{X} \arrow[hookrightarrow]{u}
    \end{tikzcd}\qquad
    \begin{tikzcd}[column sep=1.8cm]
    A' \arrow[hookrightarrow,outer sep=-1pt]{r}{\text{attach $f^n$}}
    & X' \\
    \widehat{A}' \arrow[hookrightarrow,outer sep=-1pt]{r}{\text{attach $f^n$}} \arrow[hookrightarrow]{u}
    & \widehat{X}' \arrow[hookrightarrow]{u}
    \end{tikzcd}
  \end{center}
  Let $\varphi\colon S^{n-1}\to A$ be the attaching map of the cell $e^n$.
  By naturality of the Morse complex (Theorem \ref{thm:discrete-morse-theory}), $\h X$ is obtained from $\h A$ by attaching the same cell $e^n$, so the image of $\varphi$ is actually contained in $\h A$. Thus we have $\varphi\colon S^{n-1}\to \h A$.
  With the same argument we can deduce that the attaching map $\varphi'$ of $f^n$ has image contained in $\h A'$.
  Setting $M_1 = M|_{T\times T}$ we have that $M_1$ is a Coxeter matrix corresponding to a finite Coxeter group, because $T\in S^f$.
  The CW complex $\h X = Y_{\F^*} \simeq BA_T^+$ is a space of type $K(A_T,1)$ by Corollary \ref{cor:finite-type-classifying-space}. Similarly, $\smash{\h X' = \bsal_{\F^*}(M) \simeq \bsal(M_1)}$ is also a space of type $K(A_T,1)$ by Theorem \ref{thm:conjecture-finite-type}.
  
  By induction we know that $\vartheta\colon A \to A'$ and $\vartheta'\colon A' \to A$ are homotopy inverses to each other through the homotopies $G$ and $G'$, and that the restrictions $\lambda = \vartheta|_{\h A} \colon \h A \to \h A'$ and $\lambda' = \vartheta'|_{\h A'} \colon \h A' \to \h A$ are homotopy inverses to each other through the restricted homotopies $H = G|_{\h A \times [0,1]}\colon \h A \times [0,1] \to \h A$ and $H' = G'|_{\h A' \times [0,1]}\colon \h A' \times [0,1] \to \h A'$.
  
  Since $\pi_{n-1}(\h X) = 0$ and $\pi_{n-1}(\h X') = 0$ (basepoints are omitted since they are irrelevant here), the maps $\lambda$ and $\lambda'$ can be extended to maps
  \[ \tilde \lambda \colon \h X \to \h X', \quad \tilde \lambda' \colon \h X' \to \h X. \]
  Extend $H$ to the map $\bar H \colon \h A \times [0,1] \cup \h X \times \{0, 1\}\to \h X$ where $\bar H|_{\h X \times \{0\}} = \tilde \lambda' \circ \tilde \lambda \colon \h X \to \h X$ and $\bar H|_{\h X \times \{1\}} = \id_{\h X}$.
  Then attaching the $(n+1)$-cell $e^n \times [0,1]$ to $\h A \times [0,1] \cup \h X \times \{0, 1\}$ yields the space $\h X \times [0,1]$, and $\bar H$ can be extended to a map $\smash{\tilde H \colon \h X \times [0,1] \to \h X}$ because $\smash{\pi_n(\h X)=0}$.
  By construction, the map $\smash{\tilde H}$ is a homotopy between $\smash{\tilde \lambda' \circ \tilde \lambda}$ and the identity map of $\smash{\h X}$.
  Extend similarly $H'$ to a homotopy $\smash{\tilde H'}$ between $\smash{\tilde \lambda \circ \tilde \lambda'}$ and the identity map of $\smash{\h X'}$.
  
  Finally extend $\vartheta$ to $e^n$ by gluing it with $\tilde \lambda$ (these two maps coincide on $A \cap \h X = \h A$), extend $G$ to $e^n\times [0,1]$ by gluing it with $\tilde H$ (these two homotopies coincide on $A\times [0,1] \cap \h X\times [0,1] = \h A \times [0,1]$), and do the same for $\vartheta'$ and $G'$.
  Call $\tilde\vartheta$, $\tilde G$, $\tilde \vartheta'$ and $\tilde G'$ the extended maps.
  By construction $\tilde G$ is a homotopy between $\tilde\vartheta' \circ \tilde\vartheta$ and $\id_X$, and $\tilde G'$ is a homotopy between $\tilde\vartheta \circ \tilde\vartheta'$ and $\id_{X'}$.
  
  To complete our induction argument we only need to prove that, for any subset $\F\subseteq \F_{i+1}$ closed under inclusion, the restrictions $\tilde \vartheta|_{Y_\F}$ and $\smash{\tilde G|_{Y_\F\times [0,1]}}$ have image contained in $\bsal_\F(M)$ and $Y_\F$, respectively (the analogous property for $\tilde \vartheta'$ and $\tilde G'$ will hold similarly).
  If $T\not\in \F$ then $\F \subseteq \F_{i}$, so $\tilde\vartheta|_{Y_\F}$ and $\tilde G|_{Y_\F \times [0,1]}$ are restrictions of $\vartheta$ and $G$, and our claim follows by induction.
  So we can assume $T\in\F$.
  If we set $\F' = \F \setminus\{T\}$ then $\F' \subseteq \F_{i}$, so the claim holds for $\F'$ by induction.
  If $e^n$ is the cell corresponding to $T$, then by construction the restrictions $\tilde \vartheta|_{e^n}$ and $\tilde G|_{e^n\times [0,1]}$ have image contained in $\bsal_{\F^*}(M)$ and $Y_{\F^*}$, respectively, where $\F^* = \{ R \in \sf \mid R\subseteq T \}$. Since $\F$ is closed under inclusion and $T\in \F$, then $\F^* \subseteq \F$. Therefore the restrictions $\tilde \vartheta|_{Y_\F}$ and $\smash{\tilde G|_{Y_\F\times [0,1]}}$ have image contained in $\bsal_{\F'}(M) \cup \bsal_{\F^*}(M) = \bsal_\F(M)$ and $Y_{\F'} \cup Y_{\F^*} = Y_\F$, respectively.
\end{proof}

In her paper \cite{ozornova2017discrete} Ozornova found an algebraic complex that computes the homology of Artin groups through the algebraic version of discrete Morse theory, and left open the question of whether this complex is related to the complex that computes the cellular homology of the Salvetti complex.
The homotopy equivalence $\psi$ of Theorem \ref{thm:equivalence-salvetti-morse} induces isomorphisms
\[ \psi_* \colon H_k(Y_k, Y_{k-1}; \Z) \to H_k(\bsal(M)_k,\, \bsal(M)_{k-1};\, \Z) \]
between the two cellular chain complexes, for all dimensions $k$ (here the notation $C_k$ indicates the $k$-skeleton of the CW complex $C$).
By the naturality property proved in Theorem \ref{thm:equivalence-salvetti-morse}, $\psi_*$ sends (up to sign) the standard generators of $\smash{H_k(Y_k, Y_{k-1}; \Z)}$, namely the $k$-cells of $Y$, to the corresponding $k$-cells of $\bsal(M)$ seen as generators of $\smash{H_k(\bsal(M)_k,\, \bsal(M)_{k-1};\, \Z)}$. This can be easily seen by induction, attaching one cell at a time, and using the fact that $\psi$ is a homotopy equivalence before and after adding each cell.
Moreover, by naturality of the cellular boundary maps, the following diagram is commutative for all $k$:
\begin{center}
  \begin{tikzcd}
    H_k(Y_k, Y_{k-1}; \Z) \arrow{r}{d_k} \arrow{d}{\psi_*}
    & H_{k-1}(Y_{k-1}, Y_{k-2};\Z) \arrow{d}{\psi_*} \\
    H_k(\bsal(M)_k,\, \bsal(M)_{k-1};\, \Z) \arrow{r}{d'_k}
    & H_{k-1}(\bsal(M)_{k-1},\, \bsal(M)_{k-2};\, \Z).
  \end{tikzcd}
\end{center}
Therefore $\psi_*$ is an isomorphism of algebraic complexes.
In other words we have proved that Ozornova's algebraic complex coincides with the algebraic complex which computes the cellular homology of the Salvetti complex, as it was natural to expect since these two complexes have the same rank in all dimensions.

\section{Acknowledgements}

This work is based on the material of my master thesis, written under the supervision of Mario Salvetti. Therefore I would like to thank him for introducing me to the topic and for giving me good advice throughout my final year of master degree at the University of Pisa.

I would also like to thank the referee, for his careful reading and useful comments.

\bibliography{bibliography}{}
\bibliographystyle{amsalpha-abbr}

\end{document}